\documentclass[11pt]{amsart}
\usepackage{geometry}                
\usepackage{graphicx}
\usepackage{amssymb}
\usepackage{epstopdf}
\newtheorem{theorem}{Theorem}
\newtheorem{corollary}[theorem]{Corollary}
\title[A local spectral condition for strong compactness] {A local spectral condition for strong compactness with some applications to  bilateral weighted shifts}
\author{Miguel Lacruz and Mar\'{\i}a del Pilar Romero de la Rosa}
\begin{document}
\renewcommand{\thefootnote}{}
\begin{abstract}
An algebra of bounded linear operators on a Banach space is said to be {\em strongly compact} if its unit ball is precompact in the strong operator topology, and a bounded linear operator on a Banach space is said to be {\em strongly compact} if the algebra with identity generated by the operator is strongly compact. Our interest in this notion  stems from  the work of Lomonosov on the existence of invariant subspaces.  We provide a local spectral condition that is sufficient for a bounded linear operator on a Banach space to be strongly compact. This condition is then applied to describe a large class of strongly compact, injective bilateral weighted shifts on Hilbert spaces, extending earlier work of Fern\'andez-Valles and the first author. Further applications are also derived, for instance, a strongly compact, invertible bilateral weighted shift is constructed in such a way that its inverse fails to be a strongly compact operator.
\end{abstract}
\footnote{This work was partially supported by Ministerio de Ciencia e Innovaci\'on under project MTM 2009-08934, and by Junta de Andaluc\'{\i}a under   grant FQM-3737 .}
\maketitle
\section{Introduction}
\label{intro}
\noindent
Let \(\mathcal{B}(X)\) denote the algebra of bounded linear operators on a complex Banach space \(X\). A subalgebra \(\mathcal{R} \subseteq \mathcal {B}(X)\) is said to be {\em strongly compact} if its unit ball \(\{R \in \mathcal{R}: \|R\| \leq 1\}\) is precompact in the strong operator topology, and an operator \(T \in \mathcal{B}(X)\) is said to be {\em strongly compact} if the algebra    with identity generated by \(T\) is strongly compact. 

Our interest in this notion stems from the work of Lomonosov \cite{lomonosov} on the existence of invariant subspaces  for essentially normal operators on Hilbert space. He showed that if \(T\) is an essentially normal operator on a Hilbert space such that both its  commutant \(\{T\}^\prime\) and the commutant  \(\{T^\ast\}^\prime\) of its adjoint fail to be  strongly compact algebras then \(T\) has a nontrivial invariant subspace, and moreover, if both \(T\) and \(T^\ast\) fail to be strongly compact operators then \(T\) has a nontrivial hyperinvariant subspace.

Lomonosov, Radjavi and Troitsky \cite{LRT} showed more recently that if \(T\) is an operator  such that its commutant \(\{T\}^\prime\) is a localizing algebra and the commutant \(\{T^\ast\}^\prime\) of its adjoint is a strongly compact algebra, then the adjoint  \(T^\ast\) has an invariant subspace.

A characterization of strongly compact, normal operators was given by Lomonosov, Rodr\'{\i}guez-Piazza, and the first author \cite{LLR} by a suitable application of the spectral theorem. Necessary and sufficient conditions were also provided for a unilateral weighted shift to be strongly compact in terms of the sliding products of its weights. Some applications were derived, for instance, the restriction of a strongly compact operator to an invariant subspace need not be a strongly compact operator.

Prunaru \cite{prunaru} used a particular case of a more general result obtained by Lomonosov \cite{lomonosov} to show that for a pure hyponormal, essentially normal operator \(T\) on a Hilbert space, the commutant \(\{T^\ast\}^\prime\) of its adjoint  is a strongly compact algebra.

Rodr\'{\i}guez-Piazza and the first author \cite{LRP}  showed that the position operator on the space of square integrable functions with respect to a finite measure of compact support is strongly compact if and only if the restriction of the measure to the exterior boundary of  its support is purely atomic. (A similar result was  obtained earlier on by Froelich and  Marsalli \cite{FM} within the framework of function algebras.) Further applications were derived, for instance, the weakly closed algebra generated by a strongly compact normal operator need not be a strongly compact algebra.

A classification of operator algebras was provided by Marsalli \cite{marsalli}.  He gave some sufficient conditions for an algebra of operators to be strongly compact. He  showed  that if an operator has a set of  eigenvectors that generate a dense linear manifold, then the operator is strongly compact, and moreover, if  all  the corresponding eigenvalues   have finite multiplicity, then the commutant of the operator is  strongly compact. Quite recently, Fern\'andez-Valles and the first author \cite{FVL} applied this condition to test strong compactness for several  classes of operators, namely, Ces\`aro operators, bilateral weighted shifts, and composition operators. 

A condition of a different nature seems to be needed to prove strong compactness in absence of eigenvalues. The purpose of this paper is to provide a local spectral condition that is sufficient for an operator on a Banach space to be strongly compact.  The condition requires from the operator that the origin must lie in the interior of its full spectrum and that there must be a spanning set of vectors where the local spectral radius is smaller than the distance from the origin to the boundary of the full spectrum.  This condition can be applied to operators with no eigenvalues and it  fits like a glove to bilateral weighted shifts.

As an application, we  describe a large class of strongly compact, bilateral weighted shifts on Hilbert spaces, completing the work of Fern\'andez-Valles and the first author \cite{FVL}. Then, we restrict our attention to invertible bilateral weighted shifts and we derive a condition for such operators to have a strongly compact inverse. Next, we show that if an invertible bilateral weighted shift and its inverse are strongly compact operators then the algebra generated by both of them is strongly compact. Finally, we give an example of a strongly compact, invertible bilateral weighted shift whose inverse fails to be strongly compact.

\section{A local spectral condition for strong compactness}
\label{spectral}
\noindent 
Let \(T\) denote an operator on a complex Banach space \(X\) and let \(\sigma (T)\) denote its spectrum. The {\em spectral radius} of \(T\) is defined as \(r(T)= \max \{ |\lambda|: \lambda \in \sigma (T)\}\). The spectral radius formula provides the alternative expression
\[
r(T)=\lim_{n \rightarrow \infty} \|T^n\|^{1/n}.
\]
The {\em local spectral radius} of \(T\) at a vector \(x \in X\) is defined as
\[
r(x,T)=\limsup_{n \rightarrow \infty} \|T^nx\|^{1/n}.
\]
It is clear that \(r(x,T) \leq r(T)\) for every \(x \in X\). Dane\v{s} \cite{danes} showed that the set of vectors \(x \in X\) that satisfy the strict inequality \(r(x,T) < r(T)\) is of the first category. See also  the paper of M\"{u}ller \cite{Muller}  for more information on this remarkable fact and  other related results.

Let \(G\) denote the unbounded connected component of \(\mathbb{C} \backslash \sigma (T)\). The {\em full spectrum} of \(T\) is the set \(\eta(\sigma(T))= \mathbb{C} \backslash G\). It is plain that \(\sigma (T) \subseteq
\eta(\sigma(T))\) and that \(\mathbb{C} \backslash \eta(\sigma(T))\) is connected. Also, it follows from the maximum modulus principle that if \(p\) is any polynomial then
\[
\max \{ \,|p(\lambda)|: \lambda \in \sigma (T)\} =
\max \{ \,|p(\lambda)|: \lambda \in \eta( \sigma (T))\}.
\]

Recall that \(S \subseteq X\) is a {\em spanning set} if \(S\) generates a dense linear subspace of \(X\). It turns out that a subalgebra \({\mathcal R} \subseteq {\mathcal B}(X)\) is strongly compact if and only if there is a spanning set \(S \subseteq X\) such that \(\{Rx: R \in {\mathcal R}, \, \|R\| \leq 1\} \) is precompact for every \(x \in S\). The  main result on this paper relies on this simple fact. See the paper of Lomonosov, Rodr\'{\i}guez-Piazza and the first author \cite{LLR} for a proof of it, and see also the paper of Marsalli \cite{marsalli} for the same result under a slightly different formulation.

\begin{theorem}
\label{main} 
Let \(T\) be an operator on a complex Banach space \(X\) with \(0 \in {\rm int\,} \eta ( \sigma (T))\) and suppose that there is a spanning set \(S \subseteq X\)  such that \( r(x,T)
< d(0,\partial \eta ( \sigma (T)))\) for every \(x \in S\). Then the operator \(T\) is  strongly compact.
\end{theorem}
\begin{proof}
First, consider the positive distance \(d = d(0, \partial \eta ( \sigma (T)))= \min \{ |\lambda| : \lambda \in \partial \eta ( \sigma (T))\} \). Since \(0 \in  {\rm int\,} \eta ( \sigma (T))\), we have \(\{\lambda \in \mathbb{C} : |\lambda| \leq d \}\subseteq \eta ( \sigma (T))\). It follows from the above remarks and the spectral mapping theorem that for any polynomial \(p\) we have
\begin{eqnarray*}
\max \{|p(\lambda)|: |\lambda|=d \} & \leq & \max \{ \,|p(\lambda)|: \lambda \in \eta( \sigma (T))\}\\ 
& = & \max \{ \,|p(\lambda)|: \lambda \in \sigma (T)\}\\
& = & r(p(T)) \leq \|p(T)\|.
\end{eqnarray*}
This inequality allows us  to control the size of the coefficients for a complex polynomial \(p(\lambda)=a_0+a_1\lambda + \cdots + a_m \lambda^m\). Indeed,  using Cauchy's integral formula for the derivatives of order \(n=0,1, \ldots, m\) gives
\[
a_n = \frac{p^{(n)}(0)}{n!} = \frac{1}{2 \pi i} \int_{|\lambda|=d} \,\frac{p(\lambda)}{\lambda^{n+1}}\,d\lambda,
\]
and from here we get the estimate
\[
|a_n| \leq \frac{\|p(T)\|}{d^n}\,.
\]
Take a vector \(x \in S\). We need to show that the set \( \{p(T)x: \textrm{\(p\) is a polynomial}, \,\|p(T)\| \leq 1\} \) is precompact. Choose some \(c>0\) with \(r(x,T)<c<d\). Since \(r(x,T)= \limsup \|T^nx\|^{1/n}\) as \(n \rightarrow \infty\), there is an \(n_0 \) such that \(\|T^nx\| \leq c^n\) whenever \(n \geq n_0\). Let \(\varepsilon >0\) and choose an \(n_1 \geq n_0\) such that
\[
\sum_{n=n_1+1}^\infty \left (\frac{c}{d} \right )^n < \varepsilon.
\]
Finally, consider the finite dimensional subspace \(F={\rm span\,} \{ x, Tx, \ldots , T^{n_1} x\}\). We claim that \(p(T)x \in F + \varepsilon B_X\). This is obvious if \(m \leq n_1\). Otherwise, \(p(T)x\) can be expressed as the sum of two terms,
\[
p(T)x = \sum_{n=0}^{n_1} a_nz^n + \sum_{n=n_1+1}^m a_nz^n.
\]
The first term  belongs to \(F\), while the second term satisfies the estimate
\[
\left \| \sum_{n=n_1+1}^m a_nT^nx \right \| \leq
\sum_{n=n_1+1}^m |a_n| \cdot \|T^nx\| \leq \sum_{n=n_1+1}^m\left (\frac{c}{d} \right )^n < \varepsilon,
\]
which completes the proof of our claim.
\end{proof}
It is easy to see that strong compactness is preserved under similarities, that is, if \(T\) is a strongly compact operator, \(C\) is invertible, and \(T=C^{-1}RC\) then \(R\) is strongly compact. Notice that the assumption in Theorem \ref{main} is also preserved under similarities. More precisely, if  \(T\) satisfies the condition in Theorem \ref{main}, \(C\) is invertible, and \(T=C^{-1}RC\), then \(\sigma(T)=\sigma(R)\), \(\{Cx: x \in S\}\) is a spanning set, and
\(r(Cx, R) \leq r(x,T) < d(0, \partial \eta (\sigma(T))) =d(0, \partial \eta (\sigma(R)))\) for every \(x \in S\), so that \(R\) also satisfies the condition in Theorem \ref{main}.

We now recall a couple of definitions from general operator theory. 
If \(T\) is any operator on a Banach space \(X\) then the {\em lower bound} of \(T\) is defined as
\[
m(T)= \inf \{ \|Tx\|: x \in X, \|x\| \leq 1\}.
\]
We also consider the quantity
\[
r_1(T)=\sup_{n \geq 1}m(T^n)^{1/n}=\lim_{n \rightarrow \infty}m(T^n)^{1/n}.
\]
It turns out that if \(T\) is invertible then \(\|T^{-1}\|=m(T)^{-1}\) and \(r(T^{-1})=r_1(T)^{-1}\).
\section{Applications to bilateral weighted shifts}
\label{shifts}
\noindent
Let \(W\) be an injective bilateral weighted shift on an infinite dimensional, separable
complex Hilbert space \(H\), that is,
\[
We_n = w_ne_{n+1},
\]
where \((e_n)\) is an orthonormal basis of \(H\), the weight sequence \((w_n)\) is bounded,
\(n\) runs through the integers, and \(w_n \neq 0\) for every \(n\). We refer to  the survey by Allen Shields \cite{shields} for information on the spectral parts of weighted shifts. 
See also the paper of  Bourhim \cite{bourhim} for local spectral properties of weighted shifts.
Now consider the quantities
\[
r^{-}(W)= \lim_{n \rightarrow \infty} \sup_{k>0} |w_{-n-k} \cdots w_{-k+1}|^{1/n}, \hskip3em
r^{+}(W)= \lim_{n \rightarrow \infty} \sup_{k \geq 0} |w_{k} \cdots w_{n+k-1}|^{1/n},
\]
\[
r_1^{-}(W)= \lim_{n \rightarrow \infty} \inf_{k>0} |w_{-n-k} \cdots w_{-k+1}|^{1/n}, \hskip3em
r_1^{+}(W)= \lim_{n \rightarrow \infty} \inf_{k \geq 0} |w_{k} \cdots w_{n+k-1}|^{1/n},
\]
\[
r_2^{-}(W)= \liminf_{n \rightarrow \infty} |w_{-1} \cdots w_{-n}|^{1/n}, \hskip3em
r_2^{+}(W)= \liminf_{n \rightarrow \infty} |w_0 \cdots w_{n-1}|^{1/n},
\]
\[
r_3^{-}(W)= \limsup_{n \rightarrow \infty} |w_{-1} \cdots w_{-n}|^{1/n}, \hskip3em
r_3^{+}(W)= \limsup_{n \rightarrow \infty} |w_0 \cdots w_{n-1}|^{1/n}.
\]
Then the following relationships are fulfilled
\[
r_1^{-}(W) \leq r_2^-(W) \leq r_3^-(W) \leq r^-(W), \hskip3em r_1^+(W) \leq r_2^+(W) \leq r_3^+(W) \leq r^+(W),
\]
\[
r(W)=\max \{r^-(W),r^+(W)\}, \hskip3em r_1(W)= \min \{r_1^{-}(W), r_1^+(W)\}.
\]

It was shown by Fern\'andez-Valles and the first author \cite{FVL} that if a bilateral weighted shift \(W\) satisfies the inequality \(r_3^+(W)<r_2^{-}(W)\) then its commutant \(\{W\}'\) is strongly compact, and it remained an open question whether or not \(W\) is strongly compact when \(r_2^-(W) \leq r_3^+(W)\). The following result provides an affirmative answer to this question for quite a large class of bilateral weighted shifts.
\begin{theorem}
\label{shift}
If a bilateral weighted shift \(W\) on an infinite dimensional, separable
complex Hilbert space satisfies the inequality \(r_3^+(W)<r(W)\) then \(W\) is a strongly compact operator.
\end{theorem}
\begin{proof}
The full spectrum of \(W\) is the closed disk of radius \(r(W)\) centered at the origin and the orthornormal basis \((e_k)\) is a spanning subset of \(H.\) Thus, it suffices to show that \(r(e_k,W)=r_3^+(W)\) for each \(k\). This is trivial when \(k =0\) because \(\|W^ne_0\|=|w_0 \cdots w_{n-1}|\) for every \(n \geq 1,\) so that
\[
r(e_0,W)=\limsup_{n \rightarrow \infty} \|W^ne_0\|^{1/n} = \limsup_{n \rightarrow \infty} |w_0 \cdots w_{n-1}|^{1/n}=r_3^+(W).
\]
Then, suppose that \(k \neq 0\) and notice that \(\|W^ne_k\|=|w_k \cdots w_{n+k-1}|\) for every \(n \geq 1.\) Now, there are two possibilities. On the one hand, if \(k>0\) then
\[
\|W^ne_k\| = \frac{|w_0 \cdots w_{n+k-1}|}{|w_0 \cdots w_{k-1}|}=\frac{\|W^{n+k}e_0\|}{\|W^ke_0\|},
\]
so that
\begin{eqnarray*}
r(e_k,W)    &=&\limsup_{n \rightarrow \infty} \|W^ne_k\|^{1/n}
             = \limsup_{n \rightarrow \infty} \left ( \frac{\|W^{n+k}e_0\|}{\|W^ke_0\|} \right )^{1/n}\\
            \\
            &=&\limsup_{n \rightarrow \infty} \|W^{n+k}e_0\|^{1/n}
             = \limsup_{n \rightarrow \infty} \left ( \|W^{n+k}e_0\|^{1/(n+k)} \right )^{(n+k)/n}=r(e_0,W).\\
\end{eqnarray*}
On the other hand, if \(k<0\) and \(n>-k\) then
\[
\|W^ne_k\| = |w_k \cdots w_{-1}| \cdot |w_0 \cdots w_{n+k-1}|=\|W^{-k}e_k\| \cdot \|W^{n+k}e_0\|,
\]
so that
\begin{eqnarray*}
r(e_k,W)    &=&\limsup_{n \rightarrow \infty} \|W^ne_k\|^{1/n}
             = \limsup_{n \rightarrow \infty} \left ( \|W^{-k}e_k\| \cdot \|W^{n+k}e_0\| \right )^{1/n}\\
            &=&\limsup_{n \rightarrow \infty} \|W^{n+k}e_0\|^{1/n}
             = \limsup_{n \rightarrow \infty} \left ( \|W^{n+k}e_0\|^{1/n+k} \right )^{(n+k)/n}=r(e_0,W),
\end{eqnarray*}
as we wanted.
\end{proof}
\begin{corollary}
\label{inverse1}
Let  \(W\) be an invertible bilateral weighted shift on an infinite dimensional, separable complex Hilbert space. If \(W\) satisfies the inequality \(r_1(W) < r_2^-(W)\) then its inverse \(W^{-1}\) is a strongly compact operator.
\end{corollary}
\begin{proof}
If \(W\) is an invertible bilateral weighted shift with sequence of weights \((w_n)\) then
\[
W^{-1}e_n=\frac{1}{w_{n-1}}e_{n-1}.
\]
Consider the unitary operator \(U\) defined by \(Ue_n=e_{-n}\), so that \(U^\ast=U=U^{-1}\). Now define \(V=U^\ast W^{-1}U\) and notice that \(V\) is another bilateral weighted shift, namely, \(Ve_n=v_ne_{n+1}\), where the sequence of weights \((v_n)\) is given by
\[
v_n=\frac{1}{w_{-(n+1)}}.
\]
According with the remarks at the end of the proof of Theorem \ref{main},  strong compactness is preserved under similarities. Hence, it suffices to show that   \(V\) is strongly compact.
Then, a quick computation yields \(r_3^+(V)=r_2^-(W)^{-1} < r_1(W)^{-1}=r(V).\)
\end{proof}
\begin{theorem}
\label{rational}
Let \(W\) be an invertible bilateral weighted shift on an infinite dimensional, separable complex Hilbert space. If both \(W\) and \(W^{-1}\) are strongly compact then the algebra generated by \(W\) and \(W^{-1}\) is strongly compact.
\end{theorem}
\begin{proof}
It suffices to show that  the set
\[
C_k=\{p(W)e_k+q(W^{-1})e_k: p,q \textrm{ are polynomials, }q(0)=0, \|p(W)+q(W^{-1})\| \leq 1\}
\]
is precompact for every integer \(k\). Since \(W\) and \(W^{-1}\) are both strongly compact, the sets
\begin{eqnarray*}
C_k^+ & = & \{p(W)e_k: \textrm{\(p\) is a polynomial},\, \|p(W)\| \leq 1\} \\
C_k^- & = & \{q(W^{-1})e_k : \textrm{\(q\) is a polynomial}, \,q(0)=0, \,\|q(W^{-1})\| \leq 1\}
\end{eqnarray*}
are both precompact. Notice that for every pair of polynomials \(p,q\) with \(q(0)=0\) we have
\(\|p(W)e_k\| \leq \|p(W)e_k+q(W^{-1})e_k\|\) and \(\|q(W^{-1})e_k\| \leq \|p(W)e_k+q(W^{-1})e_k\|\). Now it follows that \(C_k \subseteq C_k^+ + C_k^-\) so that \(C_k\) is precompact, as we wanted.
\end{proof}
\begin{corollary}
\label{inverse2}
Let \(W\) be an invertible bilateral weighted shift on an infinite dimensional, separable complex Hilbert space. If  both inequalities \(r_3^+(W)<r(W)\) and \(r_1(W) < r_2^-(W)\) are fulfilled  then the algebra generated by \(W\) and \(W^{-1}\) is strongly compact.
\end{corollary}
\begin{proof} Since \(r_3^+(W)<r(W)\), it follows from Theorem \ref{shift} that \(W\) is strongly compact, and since  \(r_1(W) < r_2^-(W)\), it follows from Corollary \ref{inverse1} that \(W^{-1}\) is also strongly compact. Hence, the desired result is  a consequence of Theorem \ref{rational}.
\end{proof}
\noindent
{\em Example.} Consider a bilateral weighted shift \(W\) whose sequence of weights \((w_n)\) satisfies the conditions
 \(1 \leq |w_n| \leq 2\) for each \(n \geq 0\) and \(w_n=1\) for each \(n<0\). It is plain that such a weighted shift is bounded and  invertible, and 	that \(\|W^{-1}\| =1\). We claim that \(W^{-1}\) fails to be strongly compact, for otherwise the sequence \(e_{-n}= W^{-n}e_0\) would have a norm convergent subsequence as \(n \rightarrow \infty\), a contradiction. Then, the sequence of weights can be chosen in such a way that
 \[
 \limsup_{n \rightarrow \infty} |w_0 \cdots w_{n-1}|^{1/n} =1 \hskip2em \text{ and } \hskip2em  \lim_{n \rightarrow \infty} \sup_{k \geq 0} |w_{k} \cdots w_{n+k-1}|^{1/n}=2.
 \]
There are many possible  choices for such a sequence of weights, for instance, if we set  \(w_k=2\) whenever \(2^n \leq k \leq 2^n+n\) and \(w_k=1\) otherwise. Therefore, \(r_3^+(W)=1\) and \(r(W)=r^+(W)=2\). It follows from Theorem \ref{shift}  that \(W\) is a strongly compact operator. Summarizing, we  constructed a strongly compact bilateral weighted shift  whose inverse fails to be strongly compact.
\vskip1ex
\noindent
{\em Acknowledgement.} We are indebted to Fernando Le\'on-Saavedra for his kind invitation to visit  Universidad de C\'adiz at Jerez de la Frontera during the preparation of this manuscript. We had many conversations that made possible the present work. We would like to express our sincere gratitude for his generosity and his hospitality.

\vskip3em
\begin{tabular}{lcl}
Miguel Lacruz & \hskip3em&  Mar\'{\i}a del Pilar Romero de la Rosa\\
Departamento de An\'alisis & \hskip3em & Departamento de Matem\'aticas\\
Universidad de Sevilla\ & \hskip3em &Universidad de C\'adiz\\
Avenida Reina Mercedes & \hskip3em &Avenida de la Universidad\\
41012 Seville (SPAIN) & \hskip3em & 11402 Jerez de la Frontera (SPAIN)\\
{\tt lacruz@us.es} & \hskip3em &{\tt pilar.romero@uca.es}
\end{tabular}

\end{document}